\documentclass[a4paper,10pt]{article}

\usepackage[utf8]{inputenc}

\usepackage{amsmath,amsfonts,amsthm}

\newtheorem{theorem}{Theorem}[section]
\newtheorem{lemma}[theorem]{Lemma}

\newtheorem{proposition}[theorem]{Proposition}

\newtheorem{corollary}[theorem]{Corollary}

\newcommand\be{\begin{equation}}
\newcommand\ee{\end{equation}}

\newcommand{\dx}{\,\text{\rm{}d}x}
\newcommand{\dt}{\,\text{\rm{}d}t}
\newcommand{\ds}{\,\text{\rm{}d}s}
\def\R{\mathbb  R}

\title{The regularity of the positive part\\
of functions in $L^2(I; H^1(\Omega)) \cap H^1(I; H^1(\Omega)^*)$\\
with applications to parabolic equations
}

\author{Daniel Wachsmuth%
\footnote{Institut f\"ur Mathematik,
Universit\"at W\"urzburg,
97074 W\"urzburg, Germany, {\tt daniel.wachsmuth@mathematik.uni-wuerzburg.de}}}

\begin{document}

\maketitle

\paragraph{Abstract.}
Let $u\in L^2(I; H^1(\Omega))$
with $\partial_t u\in L^2(I; H^1(\Omega)^*)$  be given.
Then we show by means of a counter-example that the positive part $u^+$ of $u$ has less regularity,
in particular it holds $\partial_t u^+ \not\in L^1(I; H^1(\Omega)^*)$ in general.
Nevertheless, $u^+$ satisfies an integration-by-parts formula,
which can be used to prove non-negativity of weak solutions of parabolic equations.



\section{Introduction}

In this  note, we are concerned with the regularity of the positive part of functions
from the function space
\[
 W:=\{ u \in L^2(I; H^1(\Omega)): \partial_t u\in L^2(I; H^1(\Omega)^*)\}
\]
of Bochner integrable functions.
Here, $I=(0,T)$, $T>0$, is an open interval, and $H^1(\Omega)$ denotes the usual Sobolev space on the domain $\Omega \subset \R^n$;
$\partial_t u$ denotes the weak derivative of $u$ with respect to the time variable $t\in I$.
The underlying spaces form a so-called evolution triple (or Gelfand triple) $H^1(\Omega) \subset L^2(\Omega) =   L^2(\Omega)^* \subset H^1(\Omega)^*$
with continuous and dense embeddings.
In the sequel, we will use the commonly applied abbreviations
\[
V:=H^1(\Omega), \quad H:=L^2(\Omega).
\]
For an introduction to these kind of function spaces and their various properties, we refer to e.g.\@
\cite[Section IV.1]{gagroza}, \cite[Section 7.2]{roubicek}, \cite[Chapter 25]{wloka}.

Let $u\in W$ be given. Let us denote its positive part by $u^+$,
\[
u^+(t,x) = \max( u(t,x),\ 0), \ t\in I, \ x\in \Omega.
\]
Due to the embedding $W\hookrightarrow L^2( I\times \Omega)$, the positive part is well-defined.
Moreover, since the mapping $u\mapsto u^+$ is bounded from $H^1(\Omega)$ to $H^1(\Omega)$, it
follows that for $u\in W$ also $u^+\in L^2( I; V)$ holds.
Here, the question arises whether $u\in W$ also implies $u^+\in W$.
The aim of the short note is to provide an counter-example of this claim, see Theorem \ref{theo210}.
Nevertheless, the following integration-by-parts formula holds true for all $u\in W$
\be\label{eqpi}
 \int_I \langle u_t(s), u^+(s)\rangle_{V^*,V} \ds = \frac12\|u^+(T)\|_H^2 - \frac12\|u^+(0)\|_H^2,
\ee
which enables us to show positivity of weak solutions of linear parabolic equations, see Section \ref{secpara}.


\section{The regularity of the positive part}

In this section, we study the mapping properties of $u\mapsto u^+$.
First, let us state the following well-known results:

\begin{proposition}\label{prop1}
 The mapping $u\mapsto u^+$ is Lipschitz continuous as mapping from $H$ to $H$.
 Furthermore it is bounded from $V$ to $V$, and for $u\in V$ it holds
 \[
  \nabla u^+(x) = \begin{cases} \nabla u(x) & \text{ if } u(x)>0\\  0 & \text{ if } u(x)\le 0 \end{cases}, \ x\in \Omega,
 \]
which implies $\|u^+\|_V \le \|u\|_V$.
\end{proposition}

The following result is an obvious consequence.

\begin{corollary}
 Let $u\in W$ be given. Then $u^+ \in L^2(I;V) \cap C(\bar I; H)$, and it holds
 \[
  \|u^+\|_{L^2(I;V)},\ \|u^+\|_{C(\bar I; H)}\ \le\ \|u\|_W.
 \]
\end{corollary}

With the same arguments that are classically used to proof Proposition \ref{prop1}, one can prove

\begin{corollary}
 Let $u\in W$ be given with $u_t \in L^2(I;H)$. Then $u^+ \in W$ with $u^+_t \in L^2(I;H)$.
\end{corollary}

Moreover, in this case, we have $\partial_t u^+\in L^2(Q)$, and we can write for almost all $(t,x)\in Q$
\be\label{eq002}
 \partial_t u^+(t,x) = \begin{cases} \partial_t u(t,x) & \text{ if } u(t,x)>0\\  0 & \text{ if } u(t,x)\le 0. \end{cases}
\ee
Now, if $\partial_t u$ is in $L^2(I; V^*)$ only, the representation \eqref{eq002} makes no sense, as $\partial_t u(t,\cdot)$
is only in $H^1(\Omega)^*$ for almost all $t$.

In the following, we will construct a function $u\in W$ with $\partial_t u\not\in L^2(I;H)$ such that
$\partial_t u^+ \not\in L^2(I;V^*)$.
The key idea is the observation that the mapping $u\mapsto u^+$ for $u\in L^2(\Omega)$ is {\em not} bounded
as mapping from $H^1(\Omega)^*$ to $H^1(\Omega)^*$.

To see this, set $\Omega=(0,1)$. Let us define $\psi_n(x)=\sin(2\pi n x)$.
Then it is well-known that $\psi_n$ converges weakly to zero in $L^2(\Omega)$, thus strongly to zero in $H^1(\Omega)^*$.
However, a short computation shows that
\[
 \int_0^1 \psi_n^+(x)\dx =  \int_0^1 \psi_1^+(x)\dx = \int_0^{1/2} \sin(2\pi x)dx = \frac1\pi \ne 0,
\]
which implies that $\psi_n^+$ converges weakly to the constant function $\hat\psi(x)=1/\pi$ in $L^2(\Omega)$.
Hence, $\psi_n^+$ cannot converge to zero in $H^1(\Omega)^*$.

In the sequel, we will equip
$V$ with the scalar product $(u,v)_V:=\int_\Omega \nabla u\cdot \nabla v + u \cdot v\dx$ and the associated norm.
The space $H$ is equipped with the standard $L^2(\Omega)$ inner product and norm.
We consider  the family of functions
\be\label{defpsin}
\psi_n(x):= \cos (n\pi x), \ x\in \Omega
\ee
for $n\in \mathbb N$.
Now, we will derive quantitative estimates of the norm of $\psi_n$ in $V$, $H$, and $V^*$ for $n\to\infty$.

\begin{lemma}\label{lemx21}
Let $n\in \mathbb N$ be given.
Then it holds
\[
\|\psi_n\|_{V} = \left(\frac{n^2\pi^2+1}2\right)^{1/2}\le n\pi, \quad
\|\psi_n\|_{H} = \frac1{\sqrt2},
 \quad  \|\psi_n\|_{V^*} \le \frac{1}{\sqrt2\,n\pi}
\]
\end{lemma}
\begin{proof}
 The first two identities can be verified with elementary calculations. To prove the third, consider the solution $z\in V$
 of $(z,v)_V = (\psi_n,v)_H$ for all $v\in V$. Then it follows $\|\psi_n\|_{V*}  =\|z\|_V$.
 The function $z$ is given by $z=\frac1{n^2\pi^2+1}\psi_n$, and hence
 the third estimate follows from the first.
\end{proof}

Let us show that the $V^*$-norm of $\psi_n^+$ is bounded away from zero.

\begin{lemma}\label{lemx22}
There is $C>0$ such that
\[
 \|\psi_n^+\|_{V^*} \ge C  \quad\forall n.
\]
\end{lemma}
\begin{proof}
Let $e\in H$ be defined by $e(x)=1$.
Then we have
\[
\begin{split}
(\psi_n^+,e)_H&=
 \int_0^1 \psi_n^+(x)\dx = \int_0^1 (\cos(n\pi x))^+\dx\\
 &=  n\int_0^{1/2n} \cos(n\pi x)\dx
 =  \frac{1}{\pi}.
\end{split}
\]
Let now $v_e\in V$ be defined by $v_e(x)=\min(4x,\,1,\,4(1-x))$.
Then it holds $\|v_e-e\|_H^2 = 2\int_0^{1/4}(4x)^2 \dx = \frac{1}{6}$.
Thus, we can estimate
\[
 \langle \psi_n^+, v_e\rangle_{V^*,V} \ge (\psi_n^+,e)_H - \|\psi_n^+\|_H\|v-e_e\|_H \ge \frac1\pi - \frac1{\sqrt{12}}
 = 0.0296\dots \ge \frac15.
\]
Here, we used $ \|\psi_n^+\|_H \le \|\psi_n\|_{H} =1/\sqrt2 $.
The lower bound implies that $\|\psi_n^+\|_{V^*} \ge \frac15 \|v_e\|_V^{-1}$, and the claim is proven.
\end{proof}

Let us now introduce a family of functions on small time intervals, which will be used to define
the counterexample by means of an infinite series.

\begin{lemma}\label{lemx23}
Let $I:=(0,1)$.
 Let $\phi\in H_0^1(I)$ be given. Define
\be\label{defphin}
 \phi_n(t):=n(n+1) \cdot \phi(n(n+1)t-n).
 \ee
 Then it holds $\operatorname{supp}\phi_n\subset \left(\frac1{n+1},\frac1n\right)$ and
 \begin{align*}
  \|\phi_n\|_{L^1(I)} &= \|\phi\|_{L^1(I)}, &
  \|\partial_t\phi_n\|_{L^1(I)}& \ge n^2 \|\partial_t\phi\|_{L^1(I)}, \\
  \|\phi_n\|_{L^2(I)} &\le \sqrt2 n \|\phi\|_{L^2(I)}, &
  \|\partial_t\phi_n\|_{L^2(I)} &\le \sqrt2 n^3 \|\partial_t\phi\|_{L^2(I)},
\end{align*}
  \end{lemma}
\begin{proof}
 This follows by elementary calculations.
\end{proof}

Let us now define the function
\be\label{deftheta}
 u(x,t) = \sum_{n=1}^\infty  n^{-3}  \phi_n(t) \psi_n(x).
\ee

\begin{theorem}\label{theo210}
Let $\phi\in H_0^1(I) \setminus\{0\}$ be given with $\phi\ge0$.
Then the function $u$ defined in \eqref{deftheta} with $\psi_n$ and $\phi_n$ from \eqref{defpsin} and \eqref{defphin},
respectively, belongs to $W$.
However, the time derivative of its positive part $\partial_tu^+$ does not belong to $L^1(I;V^*)$.
\end{theorem}
\begin{proof}
Let us define the partial sum $u_N:=\sum_{n=1}^\infty \phi_n(t) \psi_n(x)$.
We will exploit the fact that the supports of the functions  $\phi_n$ are distinct.
From the Lemmas \ref{lemx21}, \ref{lemx22}, and \ref{lemx23}, we have
\[
\|u_N\|_{L^2(I;V)}^2 = \sum_{n=1}^N n^{-6} \|\phi_n\|_{L^2(I)}^2 \|\psi_n\|_V^2 \le c \sum_{n=1}^N n^{-6}\cdot n^2\cdot n^2 = c \sum_{n=1}^N n^{-2},
\]
\[
\|\partial_tu_N\|_{L^2(I;V^*)}^2 = \sum_{n=1}^N n^{-6}\|\partial_t\phi_n\|_{L^2(I)}^2 \|\psi_n\|_{V^*}^2 \le c \sum_{n=1}^N n^{-6}\cdot n^6\cdot n^{-2} = c \sum_{n=1}^N n^{-2},
\]
\[
\|\partial_tu_N^+\|_{L^1(I;V^*)} = \sum_{n=1}^N n^{-3}\|\partial_t\phi_n\|_{L^1(I)} \|\psi_n^+\|_{V^*} \ge c \sum_{n=1}^N  n^{-3}  \cdot n^2 \cdot 1 = c \sum_{n=1}^N n^{-1}.
\]
This proves that $(u_N)$ strongly converges in $W$ to $u$.
Since $u=u_N$ on $\left(\frac1{n+1},1\right)$, the weak derivative $\partial_tu^+$ exists almost everywhere on $I$, and belongs to the space $L^1_\mathrm{loc}(I; V^*)$.
Suppose that $\partial_tu^+\in L^1(I;V^*)$ holds. Then by the continuity of the integral it follows
\[
 \|\partial_tu^+\|_{L^1(I;V^*)} = \lim_{N\to\infty} \int_{1/(N+1)}^1 \|\partial_tu^+(t) \|_{V^*}\dt = \lim_{N\to\infty}\|\partial_tu_N\|_{L^1(I;V^*)} \to \infty,
\]
which is a contradiction, hence $\partial_tu^+\not\in L^1(I;V^*)$.
\end{proof}

\section{Positivity of weak solutions to parabolic equations}
\label{secpara}

Let $\Omega\subset \R^n$ be a domain.
Again, we make use of the evolution triple $V=H^1(\Omega)$, $H=L^2(\Omega)$, $V^*=(H^1(\Omega)^*)$.
Due to the counter-example in the previous section,
we cannot apply the well-known integration-by-parts results for functions in $W$ to $u^+$.
In order to prove formula \eqref{eqpi}, we recall the following density result

\begin{proposition}\cite[Lemma 7.2]{roubicek}
The space $C^\infty([0,T], V)$ is dense in $W$.
\end{proposition}

First, let us prove the integration-by-parts formula for smooth $u$.

\begin{lemma}
Let $u\in W$ with $\partial_t u\in L^2(I;L^2(\Omega))$ be given. Then it holds
\be\label{eq003}
 \int_0^T \langle \partial_t u(t),\ u^+(t)\rangle_{V^*,V}\dt= \frac12\int_0^T \partial_t \|u^+(t)\|_H^2 = \frac12\left( \|u^+(t)\|_{H}^2-  \|u^+(0)\|_{H}^2\right).
\ee
\end{lemma}
\begin{proof}
Since $\partial_t u\in L^2(I;L^2(\Omega))$, it holds $\partial_t u^+\in L^2(I;L^2(\Omega))$. With the representation \eqref{eq002} it follows
\[
 \int_Q \partial_t u(x,t) u^+(x,t) \dx\dt= \int_Q \partial_t u^+(x,t)  u^+(x,t)  \dx\dt= \frac12\int_0^T \partial_t \|u^+(t)\|_H^2\dt,
\]
which proves the claim.
\end{proof}

\begin{lemma}\label{lem33}
Let $u\in W$ be given. Then it holds
\[
 \int_0^T \langle \partial_t u(t),\ u^+(t)\rangle_{V^*,V}\dt= \frac12\int_0^T \partial_t \|u^+(t)\|_H^2 = \frac12\left( \|u^+(t)\|_{H}^2-  \|u^+(0)\|_{H}^2\right).
\]
\end{lemma}
\begin{proof}
 Let $u\in W$ be given. By density, there is $(u_k)$ in $C^\infty([0,T], V)$ with $u_k \to u$ in $W$.
 By continuity of the projection, it follows $u_k^+\to u^+$ in $C([0,T],H)$.

 Moreover, the sequence $u_k^+$ is bounded in $L^2(V)$. Hence, there is a weakly converging subsequence with weak limit $\tilde u$ in $L^2(V)$.
 Due to $u_k^+\to u^+$ in $C([0,T],H)$, it follows $\tilde u = u^+$, and the whole sequence converges weakly, $u_k^+\rightharpoonup u^+$ in $L^2(V)$.

 Since $u_k$ is smooth enough, $u_k$ satisfies \eqref{eq003}.
 Moreover, the left-hand side and the right-hand side in \eqref{eq003} converge for $k\to\infty$, proving the claim.
\end{proof}

Let us remark that this result can be proven using difference quotients, see e.g.\@  \cite[Lemma 2.5]{gruen}.

The integration-by-parts formula \eqref{eqpi} can be applied to prove non-negativity of weak solutions of parabolic equations
with non-negative data.
Let $f\in L^1(I;L^2) + L^2(I; V')$ and $u_0\in H$ be given. Then $u\in W$ is a weak solution of the parabolic equation
with homogeneous Neumann boundary conditions
\be\label{eqpara}
 \partial_t u - \Delta u =f \text{ on } Q, \quad \partial_n u=0 \text{ on } I\times \partial \Omega, \quad  u(0) = u_0(x),
\ee
if the following equation is satisfied for all $v\in V$ and almost all $t\in I$
\[
 \langle \partial u(t), \, v\rangle_{V^*,V}  + \int_\Omega \nabla u(x,t)\nabla v(x) \dx
  =  \langle f(t), \, v\rangle_{V^*,V}.
  \]

\begin{theorem}
 Let $f\in L^1(I;L^2(\Omega)) + L^2(I; V^*)$ be given, with $f\ge0$, which is $\langle f,v \rangle \ge0$ for all $v\in L^2(V)\cap  C(I;H)$ with $v\ge0$.
 Let $u_0\in H$ be given with $u_0\ge0$.
 Let $u$ be a weak solution of the parabolic equation \eqref{eqpara}.
 Then it holds $u\ge 0$.
\end{theorem}
\begin{proof}
Let us denote $u^-  = -(-u)^+ \in L^2(V)\cap  C(I;H)$.
 Testing the weak formulation with $u^-$, integrating from $0$ to $t$, and using Proposition \ref{prop1} and Lemma \ref{lem33} yields
\[
\begin{split}
 0&\ge\int_0^t\langle f(s), u^-(s)\rangle_{V^*,V}\ds\\
 &= \int_0^t \langle \partial_t u(s), \, u^-(s)\rangle_{V^*,V} \ds + \int_0^t\int_\Omega \nabla u(x,s)\nabla u^-(x,s) \dx\ds\\
 &=\frac12\left( \|u^-(t)\|_{H}^2-  \|u^-(0)\|_{H}^2\right) + \|\nabla u^-\|_{L^2(0,t;L^2(\Omega))}^2\\
 &\ge\frac12\|u^-(t)\|_{H}^2 .
\end{split}
\]
Hence, it follows $u^-(t)=0$ for almost all $t\in I$, which implies $u^-=0$
almost everywhere on $Q$.
\end{proof}

\bibliography{w0t}
\bibliographystyle{plain}

\end{document}